\documentclass[a4paper,reqno,11pt]{amsart}
\usepackage{amsfonts,amsmath,amsthm,amssymb,stmaryrd}
\usepackage{natbib}

\allowdisplaybreaks[4]%跨页

\newtheorem{Theorem}{Theorem}[section]

\newtheorem{Lemma}[Theorem]{Lemma}
\newtheorem{Proposition}{Proposition}[section]
\newtheorem{Remark}{Remark}[section]

\newcommand{\pt}{\partial_{t}}

\newcommand{\pz}{\partial_{z}}

\newcommand{\q}{\quad}
\newcommand{\e}{\varepsilon}

\numberwithin{equation}{section} \allowdisplaybreaks
\usepackage[left=1 in, right=1 in,top=1 in, bottom=1 in]{geometry}
\begin{document}
\bibliographystyle{plain}
\title[Linear ill-posedness of three-dimensional MHD boundary layer equations]{\bf  Linear ill-posedness of three-dimensional MHD boundary layer equations}
\author{Mingxue Zhang$^{1}$}
\author{Zhonger Wu$^{2,*}$}
\thanks{$^{1}$ Institute for Math and AI, Wuhan University, Wuhan, 430072, China }
\thanks{$^{2}$ Department of Mathematics, Shantou University, Shantou 515063, China }
\thanks{$^{*}$Corresponding author: Zhonger Wu, wze622520@163.com}
\thanks{Mingxue Zhang, zhangmingxuex@163.com}

\begin{abstract}
%We study the linear ill-posedness of a three-dimensional resistive MHD boundary layer system in tangential Sobolev spaces. The equations are linearized around a time-dependent shear profile $(\mathbf U_s(t,z),\mathbf B_s(t,z))$, where both the velocity and the magnetic shear satisfy one-dimensional heat equations. We assume that a suitable linear combination of the two tangential velocity components has a non-degenerate critical point and that both tangential components of the initial magnetic shear vanish to first order at the same point. By combining the high-frequency critical-layer construction for the three-dimensional Prandtl equations with a vector-valued magnetic quotient adapted from the two-dimensional MHD boundary layer problem, we construct rapidly growing approximate solutions and establish the required residual estimates. This construction yields linear ill-posedness for every loss of fewer than $\frac18$ of a tangential derivative. The result shows how the three-dimensional Prandtl instability may reappear when the stabilizing tangential magnetic field degenerates.
We prove linear ill-posedness in tangential Sobolev spaces for the three-dimensional resistive MHD boundary layer equations. We assume that a suitable linear combination of the initial tangential velocity components has a non-degenerate critical point and that both initial tangential magnetic components have a double zero at the same point. Compared with the corresponding two-dimensional MHD result, our construction requires only this double-zero condition, rather than a higher-order degeneracy of the magnetic shear. The proof combines the three-dimensional Prandtl critical layer construction with a vector correction to the two-dimensional magnetic quotient. This correction cancels the leading terms in both induction equations while preserving the magnetic divergence constraint. The resulting approximate solutions grow like $\exp(\sigma t/\sqrt{\varepsilon})$, and their residuals have a prefactor of order $O(\varepsilon^{-3/8})$, up to logarithmic factors. A Duhamel argument then yields linear ill-posedness for every tangential derivative loss $\mu<\frac18$. Thus the three-dimensional Prandtl instability persists when the stabilizing tangential magnetic field degenerates.

 \bigbreak

\end{abstract}
\maketitle
\section{Introduction.}
In this paper, we consider the three-dimensional MHD boundary layer equations on $\Omega=\{(t,x,y,z): t>0,(x,y)\in \mathbb{T}^2,z\in\mathbb{R}^{+}\}$, namely
\begin{equation}\label{mhd}
\left\{
\begin{aligned}
 &(\partial_t+\mathbf u\cdot\nabla_{\!h}
   +w\partial_z-\partial_z^2)\mathbf u
   -(\mathbf b\cdot\nabla_{\!h}+h\partial_z)\mathbf b
   +\nabla_{\!h}p=0,\\
 &(\partial_t+\mathbf u\cdot\nabla_{\!h}
   +w\partial_z-\partial_z^2)\mathbf b
   -(\mathbf b\cdot\nabla_{\!h}+h\partial_z)\mathbf u=0,\\
 &\nabla_{\!h}\cdot\mathbf u+\partial_z w=0,
 \qquad
 \nabla_{\!h}\cdot\mathbf b+\partial_z h=0,\\
 &(\mathbf u,w)|_{z=0}=\mathbf0,
 \qquad
 (\partial_z\mathbf b,h)|_{z=0}=\mathbf0,\\
 &\lim\limits_{z\rightarrow+\infty}(\mathbf u,\mathbf b)=
   (\mathbf U^E,\mathbf B^E),
\end{aligned}
\right.
\end{equation}
where the tangential velocity and magnetic field are denoted by
$
 \mathbf u=(u,v),~\mathbf b=(b,g),
$
respectively, and $\nabla_{\!h}=(\partial_x,\partial_y)$ is the tangential gradient.
The scalar functions $w$ and $h$ denote the normal velocity and magnetic field components, respectively.
The prescribed outer-flow vectors
$
 \mathbf U^E=(U^E,V^E),~\mathbf B^E=(B^E,G^E),
$
together with the pressure $p$, satisfy the following Bernoulli's law
\begin{equation}\label{bnl}
\left\{
\begin{aligned}
 &\partial_t\mathbf U^E
  +(\mathbf U^E\cdot\nabla_{\!h})\mathbf U^E
  -(\mathbf B^E\cdot\nabla_{\!h})\mathbf B^E
  +\nabla_{\!h}p=0,\\
 &\partial_t\mathbf B^E
  +(\mathbf U^E\cdot\nabla_{\!h})\mathbf B^E
  -(\mathbf B^E\cdot\nabla_{\!h})\mathbf U^E=0.
\end{aligned}
\right.
\end{equation}

%{\color{red} In three dimensions, variation of the tangential-flow direction with the normal variable produces linear instability through a reduction to the two-dimensional critical-layer mechanism \cite{LWY}. For flows of the form \((u,ku,w)\), however, the fixed tangential direction and monotonicity ensure local well-posedness and linear stability \cite{CWYBSD}. Thus, secondary flow is the key three-dimensional instability mechanism. Without monotonicity, well-posedness is available in analytic or Gevrey spaces \cite{HD1,LMY,LYJEMS,PZ,WWZ}.}
In the absence of a magnetic field, the MHD boundary layer equations reduce to the classical Prandtl equations.
In 1963, Oleinik initiated the rigorous mathematical study of the Prandtl equations.
%Under the monotonicity condition $(\py u>0)$, Oleinik used the Crocco transformation to prove the local well-posedness of classical solutions for the two-dimensional Prandtl equations; see \cite{Oleinik1,Oleinik}.
%This theory was later reformulated in weighted Sobolev spaces by Alexandre, Wang, Xu and Yang \cite{AWXT}, using energy estimates together with a Nash--Moser iteration, and by Masmoudi and Wong \cite{MW}, using a direct nonlinear energy method. These works show that monotonicity provides a cancellation mechanism that compensates for the loss of tangential derivatives.
Under the monotonicity condition $\partial_z u>0$, Oleinik used the Crocco transformation to establish local well-posedness of classical solutions to the two-dimensional Prandtl equations; see \cite{Oleinik1,Oleinik}. Later,
Alexandre, Wang, Xu and Yang \cite{AWXT} and Masmoudi and Wong \cite{MW} obtained corresponding results directly in weighted Sobolev spaces, without using the Crocco transformation. The former combined energy estimates with a Nash-Moser iteration, whereas the latter developed a direct nonlinear energy method. In both approaches, monotonicity yields the cancellation needed to control the loss of tangential derivatives.
%This result was subsequently reproved in weighted Sobolev spaces by energy methods \cite{AWXT,MW}.
The monotonicity assumption is closely related to stability. In contrast, a non-degenerate critical point of the velocity shear produces strong linear ill-posedness in Sobolev spaces \cite{JAMS}.
%Under the same monotonicity assumption,  \cite{AWXT} and \cite{MW} independently used energy methods to prove the local well-posedness of the Prandtl equations in weighted Sobolev spaces.
%Furthermore, by imposing an additional favorable pressure condition $(\px p\leq 0)$, Xin and his collaborators \cite{XZ,XZ2024} employed the Crocco transformation to establish the global existence and uniqueness of weak solutions.
%All these results require the monotonicity assumption on the velocity field, indicating that this assumption is crucial for the well-posedness of the Prandtl equations in the framework of Sobolev spaces.
%In \cite{JAMS}, it is shown that the two-dimensional Prandtl system is linearly ill-posed in Sobolev spaces when the background shear flow has a non-degenerate critical point.
%Later, Liu and Yang \cite{LYBSD} extended this result to cases where $u_s$ converges to $\underline{u}$ without exponential convergence.
For the three-dimensional Prandtl equations, the presence of secondary flow leads to linear instability. Liu, Wang and Yang \cite{LWY} showed that this instability arises from the variation of the tangential flow direction with the normal variable. Under certain structural assumptions on the solutions, local well-posedness and linear stability were established in \cite{CWYBSD}.
%{\color{blue}Liu, Wang and Yang \cite{LWY} then identified a genuinely three-dimensional instability mechanism: a shear flow is linearly unstable when the direction of its tangential velocity varies with the normal variable. In contrast, under the special structure \((u,ku,w)\), where the tangential direction is independent of the normal variable and the scalar component is monotone, Liu, Wang and Yang \cite{CWYBSD} proved local well-posedness and linear stability with respect to general smooth three-dimensional perturbations. Thus, \cite{LWY,CWYBSD} together isolate variation of the tangential-flow direction, or equivalently the presence of secondary flow, as the decisive three-dimensional mechanism.}
%Building on the linear instability results from \cite{JAMS}, Guo and Nguyen \cite{GYCPAM} established nonlinear instability of the Prandtl equations near non-monotonic shear flows in a weak sense. Subsequently, \cite{remarksaa} systematically examined the ill-posedness of both linear and nonlinear Prandtl equations in Sobolev spaces.
Without monotonicity, well-posedness is available in analytic or Gevrey spaces \cite{HD1,LMY,LYJEMS,PZ,WWZ}.

Compared with the Prandtl equations, the MHD boundary layer equations are more complex due to the additional loss of derivatives in the magnetic field.
Under the assumption that the initial tangential magnetic field is non-degenerate , \cite{LX,CJLJFA,CJLCPAM} established the local well-posedness of the two-dimensional MHD boundary layer equations in Sobolev spaces.
A recent work by Wu \cite{WZE} established the well-posedness of the three-dimensional nonlinear MHD boundary layer equations in Sobolev space.
These results indicate that the magnetic field exerts a stabilizing effect on the boundary layer equations.
Well-posedness and long-time existence have also been investigated in analytic, partially analytic, and Gevrey spaces; see \cite{CK,SXL,LY1,lxymmas,NL,TW,FX} and the references therein.
On the other hand, \cite{lxyscm} showed that sufficient degeneracy of the tangential magnetic field at a non-degenerate critical point of the tangential velocity shear leads to linear instability of the two-dimensional MHD boundary layer equations. This result indicates that the stabilizing effect of the tangential magnetic field may break down when the field degenerates. Whether an analogous instability occurs in three dimensions remains open.
%By contrast, \cite{lxyscm} proved linear instability for the two-dimensional MHD boundary layer when the tangential magnetic field degenerates sufficiently at a non-degenerate critical point of the velocity shear. This raises the question whether the same mechanism persists in three dimensions.

%The purpose of this paper is to answer this question for the three-dimensional MHD boundary layer system. We show that the linearized equations are ill-posed in tangential Sobolev spaces when a suitable linear combination of the tangential velocity components has a non-degenerate critical point and both tangential magnetic components vanish to first order at the same point. The proof combines the three-dimensional Prandtl critical-layer construction of \cite{LWY} with the magnetic cancellation mechanism developed in \cite{CJLJFA}.

The present paper establishes linear instability for the three-dimensional resistive MHD boundary layer system. Our argument combines the Prandtl critical-layer construction of \cite{LWY} with the magnetic quotient mechanism of \cite{CJLJFA}. Compared with the two-dimensional resistive MHD result of \cite{lxyscm}, our instability criterion substantially weakens the magnetic degeneracy assumption: instead of requiring degeneracy up to
sixth order, it only requires both tangential magnetic components to have a double zero at the velocity critical point. The main new ingredient is a vector correction that cancels the leading terms in both induction equations while preserving the magnetic divergence constraint.

For simplicity, we assume that the outer-flow vectors $\mathbf U^E$ and $\mathbf B^E$ are constant. It follows from \eqref{bnl} that $\nabla_{\!h}p=0$. Thus, system \eqref{mhd} becomes
\begin{equation}\label{mhd1}
\left\{
\begin{aligned}
 &(\partial_t+\mathbf u\cdot\nabla_{\!h}
   +w\partial_z-\partial_z^2)\mathbf u
   -(\mathbf b\cdot\nabla_{\!h}+h\partial_z)\mathbf b=0,\\
 &(\partial_t+\mathbf u\cdot\nabla_{\!h}
   +w\partial_z-\partial_z^2)\mathbf b
   -(\mathbf b\cdot\nabla_{\!h}+h\partial_z)\mathbf u=0,\\
 &\nabla_{\!h}\cdot\mathbf u+\partial_z w=0,
 \qquad
 \nabla_{\!h}\cdot\mathbf b+\partial_z h=0,\\
 &(\mathbf u,w)|_{z=0}=\mathbf0,
 \qquad
 (\partial_z\mathbf b,h)|_{z=0}=\mathbf0,\\
& \lim\limits_{z\rightarrow+\infty}(\mathbf u,\mathbf b)=
   (\mathbf U^E,\mathbf B^E).
\end{aligned}
\right.
\end{equation}
Let $(\mathbf U^s(t,z),0,\mathbf B^s(t,z),0)$ be a shear-flow solution of \eqref{mhd1}, where $\mathbf U^s:=(u^s,v^s)$ and $\mathbf B^s:=(b^s,g^s)$ satisfy the following initial-boundary value problems for the heat equations
\begin{equation}\label{us}
\left\{
\begin{aligned}
 &(\partial_t-\partial_z^2)\mathbf U^s=0,\\
 &\mathbf U^s|_{z=0}=\mathbf0,\quad
 \lim_{z\to+\infty}\mathbf U^s=\mathbf U^E,\\
 &\mathbf U^s|_{t=0}=\mathbf U_s,
\end{aligned}
\right.
\qquad
\left\{
\begin{aligned}
 &(\partial_t-\partial_z^2)\mathbf B^s=0,\\
 &\partial_z\mathbf B^s|_{z=0}=\mathbf0,\quad
 \lim_{z\to+\infty}\mathbf B^s=\mathbf B^E,\\
 &\mathbf B^s|_{t=0}=\mathbf B_s.
\end{aligned}
\right.
\end{equation}
Here $\mathbf U_s=(U_s,V_s)$ and $\mathbf B_s=(B_s,G_s)$. As $z\rightarrow+\infty$, the quantities $(\mathbf U^s-\mathbf U^E,\mathbf B^s-\mathbf B^E)$ decay rapidly to zero.

Linearizing \eqref{mhd1} around this shear flow gives
%We now linearize system \eqref{mhd1} around the shear flow $(\mathbf U^s(t,z),0,\mathbf B^s(t,z),0)$. The resulting system reads
\begin{equation}\label{mhd2}
\left\{
\begin{aligned}
 &(\partial_t+\mathbf U^s\cdot\nabla_{\!h}
   -\partial_z^2)\mathbf u
   -(\mathbf B^s\cdot\nabla_{\!h})\mathbf b
   +w\partial_z\mathbf U^s-h\partial_z\mathbf B^s=0,\\
 &(\partial_t+\mathbf U^s\cdot\nabla_{\!h}
   -\partial_z^2)\mathbf b
   -(\mathbf B^s\cdot\nabla_{\!h})\mathbf u
   +w\partial_z\mathbf B^s-h\partial_z\mathbf U^s=0,\\
 &\nabla_{\!h}\cdot\mathbf u+\partial_z w=0,
 \qquad
 \nabla_{\!h}\cdot\mathbf b+\partial_z h=0,\\
 &(\mathbf u,w)|_{z=0}=\mathbf0,
 \qquad
 (\partial_z\mathbf b,h)|_{z=0}=\mathbf0,\\
& \lim_{z\to+\infty}(\mathbf u,\mathbf b)=(\mathbf0,\mathbf0),
\end{aligned}
\right.
\end{equation}
where $\mathbf u=(u,v)$ and $\mathbf b=(b,g)$ denote
the tangential velocity and magnetic perturbations,
respectively, while $w$ and $h$ denote their normal components.

First, we introduce some function spaces. For any
$\alpha,s\geq0,\beta>0$, let
\begin{align*}
\begin{split}
&L^2_\alpha(\mathbb{R^+})\triangleq\Big\{f=f(z),e^{\alpha z}f\in L^2(\mathbb{R^+})\Big\},\\
&H^s_\alpha(\mathbb{R^+})\triangleq\Big\{f=f(z),e^{\alpha z}f\in H^s(\mathbb{R^+})\Big\},\\
&W_\alpha^{s,\infty}(\mathbb{R^+})\triangleq\Big\{f=f(z),e^{\alpha z}f\in W^{s,\infty}(\mathbb{R^+})\Big\},\\
&E_{\alpha,\beta}\triangleq\Big\{f=\sum\limits_{k_1,k_2\in \mathbb{Z}}e^{i(k_1x+k_2y)}f_{k_1,k_2}(z),
\|f_{k_1,k_2}(z)\|_{L^2_\alpha}\leq C_{\alpha,\beta}e^{-\beta\sqrt{k_1^2+k_2^2}}\Big\},\\
\end{split}
\end{align*}
and
$$
\|f\|_{H_\alpha^{s}}\triangleq\|e^{\alpha z}f\|_{H^{s}},
\|f\|_{W_\alpha^{s,\infty}}\triangleq\|e^{\alpha z}f\|_{W^{s,\infty}},
\|f\|_{E_{\alpha,\beta}}\triangleq\sup\limits_{k_1,k_2\in \mathbb{Z}}
e^{\beta\sqrt{k_1^2+k_2^2}}\|f_{k_1,k_2}\|_{L^2_\alpha}.$$\\
Note that functions in $E_{\alpha,\beta}$ possess analytic regularity in the $x,y$-variables and
$L^2$ regularity in the $z$-variable.
The following proposition establishes the well-posedness for the linearized equations \eqref{mhd2} in the analytic function spaces $E_{\alpha,\beta}$.
\begin{Proposition}\label{jiexi}
Assume that $(\mathbf U^s-\mathbf U^E,\mathbf B^s-\mathbf B^E)\in C^0(\mathbb{R^+};W_\alpha^{1,\infty}(\mathbb{R^+})\cap H_\alpha^{1}(\mathbb{R^+})),$
then there exists a constant
$\delta>0$ such that, for any $T>0$  satisfying
$\beta-\delta T>0$, and initial data $(\mathbf u _0,\mathbf b _0)\in E_{\alpha,\beta}$, the linear problem \eqref{mhd2} with initial condition $(\mathbf u,\mathbf b)\big|_{t=0}=(\mathbf u _0,\mathbf b _0)$ admits a unique solution $(\mathbf u,\mathbf b)$ satisfying
$$ (\mathbf u,\mathbf b)\in C^0\big([0,T);E_{\alpha,\beta-\delta T}\big),\q (\mathbf u,\mathbf b)(t,\cdot)\in  E_{\alpha,\beta-\delta t}. $$
\end{Proposition}
This proposition can be obtained by slightly modifying \citep[Appendix A]{lxyscm}. Hence, we omit the specific details here.

In particular, Proposition \ref{jiexi} establishes existence for initial data analytic in both tangential variables. Under additional structural assumptions, analyticity in one tangential direction can be relaxed to Sobolev regularity.
Such a result was established by Chen, Li and Yang \cite{CK} for the three-dimensional linearized MHD boundary layer system, under their magnetic non-degeneracy, regularity, and compatibility assumptions.

\begin{Remark}
The partial analyticity result in \cite{CK} is analogous to that in \cite[Proposition~2.2]{LWY}, but relies on a different underlying structure. The Prandtl result uses the normal monotonicity of a tangential velocity component, whereas the MHD result uses the non-degeneracy of a tangential magnetic field component and a cancellation mechanism involving the normal magnetic field. The degenerate magnetic shear flows considered in our ill-posedness theorem fall outside the assumptions of \cite{CK}.
\end{Remark}

Suppose that $(\mathbf u,\mathbf b)(t,\cdot)$ is a solution to the linear problem \eqref{mhd2} with the initial data $(\mathbf u,\mathbf b)\big|_{t=s}=(\mathbf u _0,\mathbf b_0)$. Define the solution operator
\begin{align}\label{Tsuanzi}
	\mathcal{T}(t,s)(\mathbf u _0,\mathbf b_0):=(\mathbf u,\mathbf b)(t,\cdot).
\end{align}
We introduce a Sobolev space
\begin{align*}
	H^m:= H^m(\mathbb{T}_{x,y}^2;L^2_\alpha(\mathbb{R}_z^+)),\q m,\alpha\geq0.
\end{align*}
Since $E_{\alpha,\beta}$ is dense in $H^m$, we can extend the operator $\mathcal{T}(t,s)$ from $E_{\alpha,\beta}$ to $H^{m}$, and define the norm of the operator $\mathcal{T}$ as
$$\|\mathcal{T}\|_{\mathcal{L}(H^{m_1},H^{m_2})}
:=\sup\limits_{(\mathbf u _0,\mathbf b_0)\in E_{\alpha,\beta}}
\frac{\|\mathcal{T}(t,s)(\mathbf u _0,\mathbf b_0)\|_{H^{m_2}}}{\|(\mathbf u _0,\mathbf b_0)\|_{H^{m_1}}}\in \mathbb{R}^+\cup\{\infty\}.$$
When $\|\mathcal{T}\|_{\mathcal{L}(H^{m_1},H^{m_2})}=\infty$, it indicates that $\mathcal{T}$ cannot be extended to a bounded operator in  $\mathcal{L}(H^{m_1},H^{m_2})$.

Then, the main result of this paper is stated as follows.
\begin{Theorem}\label{TH1}
Assume that
	\begin{align*}
		\mathbf U^s-\mathbf U^E\in \mathop{\cap}\limits_{i=0}^1 C^i(\mathbb{R_+};W_\alpha^{4-2i,\infty}(\mathbb{R}^{+})\cap H_\alpha^{4-2i}(\mathbb{R}^{+})),\\
		\mathbf B^s-\mathbf B^E \in \mathop{\cap}\limits_{i=0}^2 C^i(\mathbb{R_+};W_\alpha^{6-2i,\infty}(\mathbb{R}^{+})\cap H_\alpha^{6-2i}(\mathbb{R}^{+})).
	\end{align*}
Suppose that there exists $z_0>0$ such that
	\begin{align}\label{buxiangdeng}
		\begin{aligned}
			V'_s(z_0)U''_s(z_0)\neq U'_s(z_0)V''_s(z_0),\q
			\mathbf B_s(z_0)=\partial_z \mathbf B_s(z_0)=0.
	\end{aligned}\end{align}
After interchanging $x$ and $y$ if necessary, assume further that
\begin{align}\label{aQ}
 U_s'(z_0)\neq0,
 \qquad
 a:=\frac{V_s'(z_0)}{U_s'(z_0)}\in\mathbb Q.
\end{align}
Then there exists $\sigma>0$ such that for any $\delta>0,m\geq0$ and $\mu\in [0,\frac{1}{8})$,
	\begin{align}\label{bsdjl}
		\sup\limits_{0\leq s\leq t\leq\delta}
		\|e^{-\sigma(t-s)\sqrt{|\partial_{\mathfrak{h}}|}}\mathcal{T}(t,s)\|_{\mathcal{L}(H^{m},H^{m-\mu})}=+\infty,
	\end{align}
where the operator $\partial_{\mathfrak{h}}$ denotes the tangential derivative $\partial_x$ or $\partial_y$.
\end{Theorem}

\begin{Remark}
The loss obtained here is weaker than in the related Prandtl and MHD results. The two-dimensional Prandtl result \cite{JAMS} gives $\mu<\frac12$, whereas both the three-dimensional Prandtl theorem \cite{LWY} and the two-dimensional MHD theorem \cite{lxyscm} give $\mu<\frac14$. In the present problem, the residual $O(\e^{-3/8}L_\e^N)$ and the $\sqrt\e$ factor from Duhamel's formula yield $\mu<\frac18$.  We do not claim
that this range is optimal.
\end{Remark}
\begin{Remark}\label{adejiashe}
The rationality assumption on $a$ is used only to ensure the periodicity of the high-frequency modes on $\mathbb T^2$. Indeed, if $a=\frac pq\in\mathbb Q$ and $\varepsilon_n=(qn)^{-1}$, then
$ e^{i\varepsilon_n^{-1}(y-ax)}=e^{i(qny-pnx)}$
is periodic in both tangential variables. This mode always has a high frequency in $y$ and, when $a\neq0$, also in $x$.
\end{Remark}
\begin{Remark}\label{adejiashe}
The rationality assumption on $a$ is used only to ensure the periodicity of the high-frequency modes on $\mathbb T^2$. Indeed, if $a=\frac pq\in\mathbb Q$ and $\varepsilon_n=(qn)^{-1}$, then $e^{i\varepsilon_n^{-1}(y-ax)}=e^{i(qny-pnx)}$ is periodic in both tangential variables.

In the three-dimensional Prandtl problem \cite{LWY}, the critical point may be slightly shifted so that $\frac{V_s'}{U_s'}$ is rational, by the non-degeneracy condition and the density of $\mathbb Q$. Here such a shift would generally destroy the additional condition $\mathbf B_s=\partial_z\mathbf B_s=0$. Therefore, $a\in\mathbb Q$ is
imposed as a technical compatibility condition for the present single-mode construction on $\mathbb T^2$.
\end{Remark}
%\begin{Remark}
%Unlike the non-resistive model in \cite{CJLJFA}, magnetic diffusion makes the background magnetic shear time dependent and creates an additional induction residual. We impose the corresponding double-zero condition on both tangential magnetic components.
%\end{Remark}
\begin{Remark}
Compared with the two-dimensional resistive result in \cite{lxyscm}, the present construction avoids the sixth-order degeneracy assumption on the magnetic shear and requires only a double zero at the critical point. On the other hand, unlike the non-resistive model in \cite{CJLJFA}, magnetic diffusion makes the background magnetic shear
time dependent and produces an additional induction residual. This motivates imposing the double-zero condition on both tangential magnetic components.
\end{Remark}

\section{The linear instability mechanism.}
In the absence of a magnetic field, the system \eqref{mhd} reduces to the classical Prandtl equations.
The linearized Prandtl equations near the shear flow $(u^s,v^s,0)$ in three-dimensional space read
%For the linearized Prandtl equations near the shear flow $(u^s,v^s,0)$ in three-dimensional space:
\begin{align}\label{3dprandtl}
	\left\{\begin{array}{l}
		\partial_{t}u+(u^s \partial_{x}+v^s \partial_{y})u+w\partial_{z}u^s-\partial_{z}^{2} u=0, \\
		\partial_{t}v+(u^s \partial_{x}+v^s \partial_{y})v+w\partial_{z}v^s-\partial_{z}^{2} v=0, \\
		\partial_x u+\partial_y v+\partial_z w=0,\\
		(u,v,w)\big|_{z=0}=(0,0,0),\quad \lim\limits_{z\rightarrow\infty}(u,v)=(0,0).
	\end{array}\right.
\end{align}
Liu, Wang and Yang \cite{LWY} established the ill-posedness of \eqref{3dprandtl} in Sobolev spaces.
Their argument first freezes the background shear flow in \eqref{3dprandtl} at $t=0$, that is, they set $(u^s,v^s)=(U_s,V_s)$. Then \eqref{3dprandtl} takes the following form
\begin{align}\label{3dprandtl0}
	\left\{\begin{array}{l}
		\partial_{t}u+(U_s \partial_{x}+V_s \partial_{y})u+wU'_s-\partial_{z}^{2} u=0, \\
		\partial_{t}v+(U_s \partial_{x}+V_s \partial_{y})v+wV'_s-\partial_{z}^{2} v=0, \\
		\partial_x u+\partial_y v+\partial_z w=0,\\
		(u,v,w)\big|_{z=0}=(0,0,0),\quad \lim\limits_{z\rightarrow+\infty}(u,v)=(0,0).
	\end{array}\right.
\end{align}
%Assuming $U'_s(z_0)\neq0$,  let
Set $W_s(z):=V_s(z)-aU_s(z)$. By the definition of $a$ and the determinant
condition \eqref{buxiangdeng},
$$ W_s'(z_0)=0,\qquad W_s''(z_0)\neq0.$$
Hence $z=z_0$ is a non-degenerate critical point of $W_s$. In the sequel, $\varepsilon$ is restricted to the sequence specified in Remark~\ref{adejiashe}. We seek approximate solutions of the form
%Since the determinant condition in \eqref{buxiangdeng} ensures that $U'_s(z_0)$ and $V'_s(z_0)$ cannot vanish simultaneously, we may assume, without loss of generality, that $U'_s(z_0)\neq0$.
%Let
%\begin{align}\label{a}
%	a:=\frac{V'_s(z_0)}{U'_s(z_0)}.
%\end{align}
%Following \cite[Section~3.1]{LWY}, by perturbing $z_0$ slightly if necessary, we may assume that $a=\frac{p}{q}$ with coprime integers $p,q$ and $q>0$.  In the sequel we take
%\begin{align}\label{epsilon-sequence}
% \varepsilon=\frac1{qk},\qquad k\in\mathbb N,
%\end{align}
%so that both tangential frequencies $\varepsilon^{-1}$ and $a\varepsilon^{-1}$ are integers. Then, from \eqref{buxiangdeng}, it is easy to see that $z=z_0$  is a non-degenerate critical point of the tangential velocity $W_s(z)\triangleq V_s(z)-aU_s(z)$. Next, we consider the solutions of the following form
\begin{align}\label{uvw}
	\left\{\begin{array}{l}
		(u,v)(t,x,y,z)=e^{i\varepsilon^{-1}(y-ax+\lambda_\varepsilon t)}(u_\varepsilon,v_\varepsilon)(z),\\
		w(t,x,y,z)=-i\varepsilon^{-1}e^{i\varepsilon^{-1}(y-ax+\lambda_\varepsilon t)}w_\varepsilon(z).
	\end{array}\right.
\end{align}
Using $\eqref{3dprandtl0}_3$, we have
\begin{align}\label{w1}
	w'_\varepsilon=v_\varepsilon-au_\varepsilon.
\end{align}
Substituting \eqref{uvw} and \eqref{w1} into \eqref{3dprandtl0} yields an ODE for $w_\varepsilon$
\begin{align*}
	\left\{\begin{array}{l}
		(\lambda_\varepsilon+W_s)w'_\varepsilon-W'_sw_\varepsilon+i\varepsilon w_\varepsilon^{(3)}=0,\\
		w_\varepsilon(0)=w'_\varepsilon(0)=0.
	\end{array}\right.
\end{align*}
%For the aforementioned problem, Liu-Wang-Yang \cite{LWY} derived the following asymptotic expansions for $\lambda_\varepsilon$ and $w_\varepsilon$ in terms of $\varepsilon$
Liu, Wang and Yang \cite{LWY} obtained the expansions
\begin{align*}
	\left\{\begin{array}{l}
		\lambda_\varepsilon=-W_s(z_0)+\sqrt{\varepsilon}\tau,\\
		w_\varepsilon= H(z-z_0)\big[W_s(z)-W_s(z_0)+\sqrt{\varepsilon}\tau\big]
		+\sqrt{\varepsilon}W\Big(\frac{z-z_0}{\varepsilon^{1/4}}\Big),
	\end{array}\right.
\end{align*}
where $H(z)$  is the Heaviside function, $\tau$  is a complex number with $\mathcal{I}m~\tau<0$, and $W(Z)$solves the following ODE
\begin{align*}
	\left\{\begin{array}{l}
		(\tau+W^{''}_s(z_0)\frac{Z^2}{2})W'-W^{''}_s(z_0)ZW+iW^{(3)}=0,\q  Z\neq0,\\
		{[W]|_{Z=0}=-\tau,[W']|_{Z=0}=0,[W^{''}]|_{Z=0}=-W^{''}_s(z_0)},\\
		\lim\limits_{Z\rightarrow\pm\infty}W=0.
	\end{array}\right.
\end{align*}
Here, $[W]|_{Z=0}
=\lim\limits_{\delta_1\rightarrow0^{+}}W(\delta_1)-\lim\limits_{\delta_2\rightarrow0^{-}}W(\delta_2)$ denotes the jump of $W(Z)$ at $Z=0$.
%According to G\'{e}rard-Varet-Dormy's work in \cite{JAMS}, $\big(\tau,W(Z)\big)$ has the form
According to the work of G\'{e}rard-Varet and Dormy in \cite{JAMS}, the aforementioned ODE admits a solution $W(Z)$ that decays exponentially as $|Z|\rightarrow\infty$. Specifically, $\big(\tau,W(Z)\big)$ takes the following form
\begin{align}\label{jinsizhanshi}
	\left\{\begin{array}{l}
		\tau=\big|\frac{W^{''}_s(z_0)}{2}\big|^{\frac{1}{2}}\tilde{\tau},\\
		W(Z)=\big|\frac{W^{''}_s(z_0)}{2}\big|^{\frac{1}{2}}
		\Big[\big(\tilde{\tau}+\big|\frac{W^{''}_s(z_0)}{2}\big|^{\frac{1}{2}}Z^2\big)
		\tilde{W}\big(\big|\frac{W^{''}_s(z_0)}{2}\big|^{\frac{1}{4}}Z\big)
		-\mathrm{1}_{\mathbb{R_+}}\big(\tilde{\tau}
		+\big|\frac{W^{''}_s(z_0)}{2}\big|^{\frac{1}{2}}Z^2\big)\Big].
	\end{array}\right.
\end{align}
The complex number $\tilde\tau$ and the profile $\tilde W$ are supplied by condition (SC) in \cite[Sections~2-3]{JAMS}.
To obtain an approximate solution for \eqref{3dprandtl0}, we also need to construct the asymptotic expansion of $(u_\e,v_\e)(z)$. Based on the relationship in \eqref{w1} and the asymptotic expansion in \eqref{jinsizhanshi}, \cite{LWY} derived the following result.
\begin{Lemma}\citep[Proposition 3.2]{LWY}\label{LE3.1}
	For the large frequency $k=\frac{1}{\e}$, the approximate solutions of the
problem \eqref{3dprandtl0} can be expressed as \eqref{uvw} with
	\begin{align}\label{jinsizhanshiuv}
		\left\{\begin{array}{l}
			\lambda_\varepsilon=-W_s(z_0)+\sqrt{\varepsilon}\tau,\\
			w_\varepsilon= H(z-z_0)\big[W_s(z)-W_s(z_0)+\sqrt{\varepsilon}\tau\big]
			+\sqrt{\varepsilon}W\Big(\frac{z-z_0}{\varepsilon^{1/4}}\Big),\\
			(u_\e,v_\e)(z)\sim H(z-z_0)\big(U'_s(z),V'_s(z)\big)
			+W''\Big(\frac{z-z_0}{\varepsilon^{1/4}}\Big)\frac{1}{W^{''}_s(z_0)}\big(U'_s(z_0),V'_s(z_0)\big)\\
			\q\q\q\q\q+\e^{\frac{1}{4}}W'\Big(\frac{z-z_0}{\varepsilon^{1/4}}\Big)
			\frac{1}{W^{''}_s(z_0)}\big(U''_s(z_0),V''_s(z_0)\big),
		\end{array}\right.
	\end{align}
where the complex constant $\tau$ and function $W(Z)$ are given in \eqref{jinsizhanshi}.
\end{Lemma}
In order to construct  approximate solutions to \eqref{3dprandtl}, let
%Next, they construct approximate solutions for \eqref{3dprandtl}.Define
$$w_a^s(t,z)\triangleq v^s(t,z)-au^s(t,z),$$
where $a$ is the constant given in \eqref{aQ}, and $(u^s,v^s)$ meets the requirements of Theorem \ref{TH1}. Clearly, $z_0$ is a non-degenerate critical point of $w_a^s(0,z)$. After applying the equivalent sign normalization used in the critical-layer profile if necessary, we may assume $\partial_z^2 w_a^s(0,z_0)<0$; the instability criterion itself is independent of this sign. Let $f(t)$ be a non-degenerate critical point of $w_a^s(t,z)$, satisfying the following equation in a small time interval $(0, t_0)$,
\begin{align}\label{ft}
	\left\{\begin{array}{l}
		\pt\pz w_a^s\big(t,f(t)\big)+\pz^2w_a^s\big(t,f(t)\big)f'(t)=0,\\
		f(0)=z_0.
	\end{array}\right.
\end{align}
\begin{Lemma}\citep[Section 3.2]{LWY}\label{LE3.2}
The approximate solution for problem \eqref{3dprandtl} is defined as
%\begin{align}\label{3djinsi}
%	(u_\e,v_\e,w_\e)(t,x,y,z)=e^{i\e^{-1}(y-ax)}(U_\e,V_\e,W_\e)(t,z).
%\end{align}
\begin{align}\label{jinsizhanshiuv1}
	\left\{\begin{array}{l}
\lambda_\varepsilon=-w_a^s\big(t,f(t)\big)
		+\sqrt{\varepsilon}\big|\frac{\pz^2w_a^s\big(t,f(t)\big) }{2}\big|^{\frac{1}{2}}\tilde{\tau},\\
(u_\e,v_\e,w_\e)(t,x,y,z)=e^{i\e^{-1}(y-ax)}(U_\e,V_\e,W_\e)(t,z),\\		
		W_\varepsilon(t,z)=\e^{-1}e^{i\e^{-1}\int_0^t\lambda_\varepsilon(s)\ ds}\Big(W_\e^{reg}(t,z)+W_\e^{sl}(t,z)\Big),\\
		(U_\e,V_\e)(t,z)=ie^{i\e^{-1}\int_0^t\lambda_\varepsilon(s)\ ds}\boldsymbol{\mathcal U}_\e
	\end{array}\right.
\end{align}
with
\begin{align*}
	&(U_\e^{reg},V_\e^{reg})(t,z)=H(z-f(t))\partial_z\mathbf U^s(t,z),\\
	&W_\e^{reg}(t,z)=H\big(z-f(t)\big)\Big[w_a^s(t,z)-w_a^s(t,f(t))+\sqrt{\varepsilon}\big|\frac{\pz^2w_a^s\big(t,f(t)\big) }{2}\big|^{\frac{1}{2}}\tilde{\tau}\Big],\\
	&W_\e^{sl}(t,z)=\sqrt{\e}\varphi(z-f(t))\big|\frac{\pz^2w_a^s\big(t,f(t)\big) }{2}\big|^{\frac{1}{2}}
	W_{sl}\Big(\big|\frac{\pz^2w_a^s\big(t,f(t)\big) }{2}\big|^{\frac{1}{4}}\cdot\frac{z-f(t)}{\e^{\frac{1}{4}}}\Big),\\
	&W_{sl}(Z)=(\tilde{\tau}-Z^2)\tilde{W}(Z)-\mathrm{1}_{\mathbb{R}^+}(\tilde{\tau}-Z^2),\\
&\boldsymbol{\mathcal U}_\e(t,z)=\big(U_\e^{reg},V_\e^{reg}\big)(t,z)
		+\frac{\pz^2W_\e^{sl}(t,z)}{\pz^2w_a^s\big(t,f(t)\big)}\partial_z\mathbf U^s\big(t,f(t)\big)+\frac{\pz W_\e^{sl}(t,z)}{\pz^2w_a^s\big(t,f(t)\big)}\partial_z^2\mathbf U^s\big(t,f(t)\big).
\end{align*}
\end{Lemma}

\subsection{Construction of approximate solutions to \eqref{mhd2}.}

We now combine the Prandtl mode in Lemma~\ref{LE3.2} with the quotient construction of \cite[Section~4]{CJLJFA}.  The scalar quotient used in two dimensions cannot simply be applied component by component: doing so leaves an uncancelled leading term in the three dimensional induction equations.  The following vector version removes that
term and simultaneously preserves the magnetic divergence constraint.

Put
\begin{align}\label{AQDY}
\begin{aligned}
 A_\e(t,z)&:=\lambda_\e(t)+w_a^s(t,z),
 &Q_\e(t,z)&:=W_\e^{\rm reg}(t,z)+W_\e^{\rm sl}(t,z),\\
 \rho_\e(t,z)&:=\frac{Q_\e(t,z)}{A_\e(t,z)},
 &c^s(t,z)&:=g^s(t,z)-ab^s(t,z).
\end{aligned}
\end{align}
Although numerator and denominator in $\rho_\e$ are both of order $\sqrt\e$ in the critical layer, $\rho_\e$ is regular.  Indeed,
$$ \rho_\e=H(z-f(t))+\frac{W_\e^{\rm sl}}{A_\e},
 \qquad |A_\e(t,z)|\geq c\big(\sqrt\e+|z-f(t)|^2\big) $$
on the support of the cutoff, because $\Im\tilde\tau<0$.  The jump of the Heaviside
function is exactly cancelled by the shear-layer profile.

Using the tangential vectors defined above, write the velocity amplitudes in \eqref{jinsizhanshiuv1} as
\begin{align}\label{zfdy}
 (U_\e,V_\e)=iE_\e\boldsymbol{\mathcal U}_\e,\qquad
 W_\e=\e^{-1}E_\e Q_\e,\qquad
 E_\e(t):=\exp\!\left(i\e^{-1}\int_0^t\lambda_\e(s)\,ds\right).
\end{align}
The identity coming from incompressibility is
\begin{align}\label{-a1}
 ( -a,1)\cdot\boldsymbol{\mathcal U}_\e=\partial_zQ_\e.
\end{align}
Define the tangential magnetic amplitude by
\begin{align}\label{Bedy}
 \boldsymbol{\mathcal B}_\e
 :=\rho_\e\partial_z\mathbf B^s
 +\frac{c^s}{A_\e}
 \big(\boldsymbol{\mathcal U}_\e-\rho_\e\partial_z\mathbf U^s\big),
\end{align}
and set
\begin{align}\label{bghdy}
 (b_\e,g_\e,h_\e)(t,x,y,z):=e^{i\e^{-1}(y-ax)}(B_\e,G_\e,H_\e)(t,z)
 :=e^{i\e^{-1}(y-ax)}
 \left(iE_\e\boldsymbol{\mathcal B}_\e,
 \e^{-1}E_\e c^s\rho_\e\right).
\end{align}
Here the first entry in parentheses is a two-component vector. The magnetic divergence constraint follows directly
from the quotient structure.  Indeed, \eqref{-a1} and $Q_\e=A_\e\rho_\e$ give
\[
\begin{aligned}
 (-a,1)\cdot\boldsymbol{\mathcal B}_\e
 &=\rho_\e\partial_zc^s
 +\frac{c^s}{A_\e}
   \bigl(\partial_zQ_\e-\rho_\e\partial_zA_\e\bigr)\\
 &=\rho_\e\partial_zc^s+c^s\partial_z\rho_\e
 =\partial_z(c^s\rho_\e).
\end{aligned}
\]
Consequently,
$\partial_xb_\e+\partial_yg_\e+\partial_zh_\e=0$.
Because the cutoff is supported away from $z=0$, the magnetic amplitudes
also satisfy the boundary conditions in \eqref{mhd2}.
\begin{Remark}
The velocity construction and its determinant condition follow \cite{LWY}.
The magnetic ansatz extends the scalar quotient construction of
\cite{CJLJFA}. In three dimensions, the candidate
$\partial_z(\rho_\e\mathbf B^s)$ preserves magnetic divergence but does not
generally cancel both induction equations. Formula
\eqref{Bedy} adds the transverse correction
$\frac{c^s}{A_\e}\boldsymbol{\mathcal D}_\e
-(\partial_z\rho_\e)\mathbf B^s,
$
which is orthogonal to $(-a,1)$.
\end{Remark}

Direct calculation shows that the approximate solution $(u_\e,v_\e,w_\e,b_\e,g_\e,h_\e)$ satisfies
\begin{align}\label{jsmhd2}
	\left\{\begin{array}{l}
		\partial_{t}u_\e+(u^s \partial_{x}+v^s \partial_{y})u_\e+w_\e\partial_{z}u^s-\partial_{z}^{2} u_\e
		-(b^s \partial_{x}+g^s \partial_{y}) b_\e-h_\e \partial_{z}b^s
		=r_{1\e}, \\
		\partial_{t}v_\e+(u^s \partial_{x}+v^s \partial_{y})v_\e+w_\e\partial_{z}v^s-\partial_{z}^{2} v_\e
		-(b^s \partial_{x}+g^s \partial_{y}) g_\e-h_\e \partial_{z}g^s
        =r_{2\e}, \\
		\partial_{t}b_\e+(u^s \partial_{x}+v^s \partial_{y})b_\e+w_\e\partial_{z}b^s-\partial_{z}^{2} b_\e
		-(b^s \partial_{x}+g^s \partial_{y}) u_\e-h_\e \partial_{z}u^s
        =r_{3\e}, \\
		\partial_{t}g_\e+(u^s \partial_{x}+v^s \partial_{y})g_\e+w_\e\partial_{z}g^s-\partial_{z}^{2} g_\e
		-(b^s \partial_{x}+g^s \partial_{y}) v_\e-h_\e \partial_{z}v^s
        =r_{4\e}, \\
		\partial_x u_\e+\partial_y v_\e+\partial_z w_\e=0,\quad \partial_x b_\e+\partial_y g_\e+\partial_z h_\e=0,\\
		(u_\e,v_\e,w_\e)\big|_{z=0}=(\pz b_\e,\pz g_\e,h_\e)\big|_{z=0}=(0,0,0),\quad \lim\limits_{z\rightarrow+\infty}(u_\e,v_\e,b_\e,g_\e)=(0,0,0,0),
	\end{array}\right.
\end{align}
 where $r_{i\e}=e^{i\e^{-1}(y-ax)}R_{i\e}~(i=1,2,3,4)$. Write
$$\mathbf R_\e%=(\mathbf R_\e^u,\mathbf R_\e^b)
:=(R_{1\e},R_{2\e},R_{3\e},R_{4\e}),\q\q
\boldsymbol{\mathcal D}_\e
:=\boldsymbol{\mathcal U}_\e-\rho_\e\partial_z\mathbf U^s.
$$
Since $(U_\e,V_\e)=iE_\e\boldsymbol{\mathcal U}_\e$ and
$\partial_tE_\e=i\e^{-1}\lambda_\e E_\e$, direct substitution yields
\begin{align}\label{R12}
\begin{aligned}
(R_{1\e},R_{2\e})
&=E_\e\Big\{
i(\partial_t-\partial_z^2)\boldsymbol{\mathcal U}_\e
+\e^{-1}\big[-A_\e\boldsymbol{\mathcal U}_\e
+Q_\e\partial_z\mathbf U^s
+\frac{(c^s)^2}{A_\e}\boldsymbol{\mathcal D}_\e\big]\Big\}\\
&:=E_\e\Big\{
\mathbf P_\e
+\e^{-1}\frac{(c^s)^2}{A_\e}\boldsymbol{\mathcal D}_\e\Big\},
\end{aligned}\end{align}
where $E_\e\mathbf P_\e$ denotes the two Prandtl momentum residuals.
In the induction equations the frequency terms cancel exactly:
\[
-A_\e\boldsymbol{\mathcal B}_\e
+Q_\e\partial_z\mathbf B^s
+c^s\boldsymbol{\mathcal U}_\e
-c^s\rho_\e\partial_z\mathbf U^s=0.
\]
Consequently,
\begin{align}\label{R34}
(R_{3\e},R_{4\e})=iE_\e(\partial_t-\partial_z^2)
\boldsymbol{\mathcal B}_\e.
\end{align}

For the estimates below, write $\mathbf Z_\e=(U_\e,V_\e,B_\e,G_\e)$. The symbols $\mathcal U_\e$, $\mathcal B_\e$, and $\mathcal D_\e$ denote two-component amplitudes. We record the estimate needed below.

\begin{Proposition}\label{prop-improved-residual}
Let the assumptions of Theorem~\ref{TH1} hold and let
$L_\e:=1+|\log\e|$.  For every fixed $K_0>0$ there exist $N\geq1$, $\e_0>0$ and positive constants
$c_0,C_0,\sigma_0$, independent of $\e$, such that, for
\[
 0<\e<\e_0,\qquad 0\leq t\leq K_0\sqrt\e\,L_\e,
\]
the approximate tangential profiles satisfy
\begin{align}\label{growth-improved-mode}
 c_0e^{\sigma_0t/\sqrt\e}
 &\leq\|\mathbf Z_\e(t)\|_{L_\alpha^2}
 \leq C_0L_\e^N e^{\frac{\sigma_0t}{\sqrt\e}},\\
 \label{residual-improved-mode}
 \|\mathbf R_\e(t)\|_{L_\alpha^2}
 &\leq C_0\e^{-3/8}L_\e^N e^{\frac{\sigma_0t}{\sqrt\e}}.
\end{align}
In particular, $\|\mathbf Z_\e(0)\|_{L_\alpha^2}\leq C_0L_\e^N$.
\end{Proposition}

\begin{proof}

Put $d=\e^{1/4}$ and $Z=\frac{z-f(t)}{d}$. All estimates below are uniform for
$0\leq t\leq K_0d^2L_\e$; the power $N$ may increase from line to line.
Set
\[
\sigma_*(t):=-\left|\frac{\partial_z^2w_a^s(t,f(t))}{2}\right|^{1/2}\Im\widetilde\tau>0.
\]
Since the real-valued term $-w_a^s(t,f(t))$ does not change the modulus of
$E_\e$, formula~\eqref{jinsizhanshiuv1} gives
\begin{align}\label{modulus-E}
 |E_\e(t)|
 &=\exp\left(-\e^{-1}\int_0^t\Im\lambda_\e(s)\,ds\right)
 =\exp\left(d^{-2}\int_0^t\sigma_*(s)\,ds\right).
\end{align}
Let $\sigma_0=\sigma_*(0)$.  The regularity of the shear and the implicit-function construction of $f$ give $|\sigma_*(s)-\sigma_0|\leq Cs$.  Hence
\[
 \left|d^{-2}\int_0^t(\sigma_*(s)-\sigma_0)\,ds\right|
 \leq Cd^{-2}t^2\leq CK_0^2d^2L_\e^2=o(1).
\]
Consequently, after decreasing $\e_0$ if necessary,
\begin{align}\label{E-two-sided}
 C^{-1}e^{\frac{\sigma_0t}{d^2}}\leq |E_\e(t)|
 \leq Ce^{\frac{\sigma_0t}{d^2}}.
\end{align}
The non-degenerate critical point and $\Im\widetilde\tau<0$ give, on a
sufficiently small fixed support of the cutoff,
\begin{align}\label{A-lower-bound}
 cd^2(1+Z^2)\leq|A_\e(t,z)|\leq Cd^2(1+Z^2).
\end{align}
%Then, together with $\rho_\e-H(z-f(t))=\frac{W_\e^{\rm sl}}{A_\e}$, we have
%\begin{align}\label{defect-bounds}
% |\partial_z^j(\rho_\e-H(z-f(t)))|
% \leq Cd^{-j}L_\e^Ne^{-c|Z|},\qquad 0\leq j\leq2.
%\end{align}
%Indeed, the imaginary part controls $|Z|\lesssim1$, while the quadratic Taylor term controls large $|Z|$ in this neighborhood.
Let $\boldsymbol{\mathcal D}_\e:=\boldsymbol{\mathcal U}_\e-\rho_\e\partial_z\mathbf U^s$. Using \eqref{jinsizhanshiuv1}, we readily obtain
\begin{align}\label{defect-expansion}
\begin{aligned}
 \boldsymbol{\mathcal D}_\e
 ={}&\frac{\partial_z^2W_\e^{\rm sl}}
 {\partial_z^2w_a^s(t,f(t))}\partial_z\mathbf U^s(t,f(t))
 +\frac{\partial_zW_\e^{\rm sl}}
 {\partial_z^2w_a^s(t,f(t))}\partial_z^2\mathbf U^s(t,f(t))\\
 &-\big(\rho_\e-H(f(t))\big)\partial_z\mathbf U^s(t,z).
\end{aligned}
\end{align}
Since $W_\e^{\mathrm{sl}}$ is $d^2$ times a piecewise smooth, exponentially decaying function of $Z$,
together with $\rho_\e-H(z-f(t))=\frac{W_\e^{\rm sl}}{A_\e}$, \eqref{A-lower-bound}
and \eqref{defect-expansion} yields
\begin{align}\label{defect-bounds}
 |\partial_z^j(\rho_\e-H(z-f(t)))|+|\partial_z^j\boldsymbol{\mathcal D}_\e|
 \leq C d^{-j}L_\e^N e^{-c|Z|},\qquad 0\leq j\leq2.
\end{align}
It follows from \cite[(3.28)]{LWY} and \cite[(3.35)]{LWY} that
\begin{align}\label{prandtl-profile-estimates}
 c\leq\|\boldsymbol{\mathcal U}_\e(t)\|_{L_\alpha^2}\leq CL_\e^N,
 \qquad \|E_\e\mathbf P_\e(t)\|_{L_\alpha^2}\leq Cd^{-1}e^{\sigma_0t/d^2}\leq Cd^{-1}L_\e^N e^{\sigma_0t/d^2},
\end{align}
where $E_\e\mathbf P_\e$ denotes the two Prandtl momentum residuals.

By \eqref{buxiangdeng}, the heat equations, and $f(t)-z_0=O(t)$, we obtain
\[
 \mathbf B^s(t,f(t))=O(t),\qquad
 \partial_z\mathbf B^s(t,f(t))=O(t),\qquad
 \partial_z^2\mathbf B^s(t,f(t))=O(1).
\]
Taylor expansion in $z$ therefore gives
\begin{align}\label{magnetic-parabolic-taylor}
 |\partial_z^j\mathbf B^s|+|\partial_z^jc^s|
 \leq Cd^{2-j}L_\e(1+ Z^2),\qquad j=0,1,2,
\end{align}
on the cutoff support. Also $|\partial_tc^s|+
|\partial_t\partial_z\mathbf B^s|\leq C$.
Using \eqref{Bedy}, the quotient rule and \eqref{A-lower-bound}-\eqref{magnetic-parabolic-taylor}, for $0\leq j\leq2$, we obtain
\begin{align}\label{magnetic-layer-estimates}
 \left|\partial_z^j\left(\boldsymbol{\mathcal B}_\e-
 H(z-f(t))\partial_z\mathbf B^s\right)\right|
 =\left|\partial_z^j\left(\big(\rho_\e-
 H(z-f(t))\big)\partial_z\mathbf B^s+\frac{c^s}{A_\e}\boldsymbol{\mathcal D}_\e\right)\right|
 \leq Cd^{-j}L_\e^Ne^{-c|Z|}.
\end{align}
It then follows from \eqref{magnetic-parabolic-taylor} and \eqref{magnetic-layer-estimates} that
\begin{align}\label{Be}
 \|\boldsymbol{\mathcal B}_\e\|_{L_\alpha^2}\leq
  \|\boldsymbol{\mathcal B}_\e-
 H(z-f(t))\partial_z\mathbf B^s \|_{L_\alpha^2}+\|\partial_z\mathbf B^s\|_{L_\alpha^2}
 \leq CL_\e^N.
\end{align}

We now prove the two assertions in \eqref{growth-improved-mode}.  By
\eqref{zfdy}, \eqref{bghdy}, \eqref{E-two-sided},
\eqref{prandtl-profile-estimates}  and \eqref{Be},
\[
 \begin{split}
 \|\mathbf Z_\e(t)\|_{L_\alpha^2}
 =|E_\e(t)|\,
 \|(\boldsymbol{\mathcal U}_\e,
 \boldsymbol{\mathcal B}_\e)(t)\|_{L_\alpha^2}
 \leq CL_\e^N e^{\frac{\sigma_0t}{d^2}},
 \end{split}
\]
where \eqref{E-two-sided} was used in the last line.  The velocity component
alone and the lower bound in \eqref{prandtl-profile-estimates} give
$$
 \|\mathbf Z_\e(t)\|_{L_\alpha^2}
 \geq c|E_\e(t)|\geq c e^{\sigma_0t/d^2}.
$$

Next, we will  prove \eqref{residual-improved-mode}.
Recall \eqref{R12},
\[
 (R_{1\e},R_{2\e})=E_\e\left(\mathbf P_\e+
 d^{-4}\frac{(c^s)^2}{A_\e}\boldsymbol{\mathcal D}_\e\right).
\]
Using  \eqref{A-lower-bound}, \eqref{defect-bounds} and \eqref{magnetic-parabolic-taylor}, we obtain
\begin{align*}
\left\|d^{-4}\frac{(c^s)^2}{A_\e}
\boldsymbol{\mathcal D}_\e\right\|_{L_\alpha^2}
\leq Cd^{-2}
\left\|L_\e^N(1+ Z^2) e^{-c|Z|}\right\|_{L_z^2}
\leq Cd^{-3/2}L_\e^N.
\end{align*}
Hence, \eqref{E-two-sided} and \eqref{prandtl-profile-estimates} imply
\begin{align}\label{momentum-residual-bound}
 \|(R_{1\e},R_{2\e})(t)\|_{L_\alpha^2}
 \leq Cd^{-3/2}L_\e^Ne^{\frac{\sigma_0 t}{d^2}}.
\end{align}

Recall that
\[
(R_{3\varepsilon},R_{4\varepsilon})
=
iE_\varepsilon(\partial_t-\partial_z^2)
\boldsymbol{\mathcal B}_\varepsilon,
\qquad
\boldsymbol{\mathcal B}_\varepsilon
=
\rho_\varepsilon\partial_z\mathbf B^s
+
\frac{c^s}{A_\varepsilon}
\boldsymbol{\mathcal D}_\varepsilon.
\]
We apply the parabolic operator directly to the complete magnetic
amplitude. First, the jump relations of the shear-layer profile imply
\begin{align}\label{bjump}
[\boldsymbol{\mathcal B}_\varepsilon]_{z=f(t)}=0,
\qquad
[\partial_z\boldsymbol{\mathcal B}_\varepsilon]_{z=f(t)}=0.
\end{align}
Consequently,
$(\partial_t-\partial_z^2)\boldsymbol{\mathcal B}_\varepsilon$
contains neither a Dirac mass nor a derivative of a Dirac mass at
$z=f(t)$.\\
It remains to estimate the regular part.  Since
$\rho_\e-H(z-f(t))=W_\e^{\rm sl}/A_\e$, the quotient rule and
\eqref{A-lower-bound}-\eqref{defect-bounds} and \eqref{Be} give, for
$0\leq j\leq2$,
\begin{align}\label{magnetic-correction-spatial-estimates}
 &\left|\partial_z^j\big(\rho_\e-H(z-f(t))\big)\right|
  +|\partial_z^j\boldsymbol{\mathcal D}_\e|
  \leq Cd^{-j}L_\e^Ne^{-c|Z|},\notag\\
 &\left|\partial_z^j\left[
  \big(\rho_\e-H(z-f(t))\big)\partial_z\mathbf B^s
  +\frac{c^s}{A_\e}\boldsymbol{\mathcal D}_\e
  \right]\right|
  \leq Cd^{-j}L_\e^Ne^{-c|Z|}.
\end{align}

The definition of $A_\e$ in \eqref{AQDY} and $\partial_z w_a^s(t,f(t))=0$ imply $|\partial_tA_\e|\leq C(d|Z|+d^2).$
A direct computation yields
\[
\pt\Big(\rho_\e-H(z-f(t))\Big)=\frac{\pt W_\e^{\rm sl}}{A_\e}-\frac{W_\e^{\rm sl}\pt A_\e}{A^2_\e},
\]
and
\begin{align*}
\partial_t\left(
\frac{c^s}{A_\varepsilon}\boldsymbol{\mathcal D}_\varepsilon
\right)
=
\frac{\partial_tc^s}{A_\varepsilon}\boldsymbol{\mathcal D}_\varepsilon
&-\frac{c^s\partial_tA_\varepsilon}{A_\varepsilon^2}
\boldsymbol{\mathcal D}_\varepsilon
+\frac{c^s}{A_\varepsilon}
\partial_t\boldsymbol{\mathcal D}_\varepsilon.
\end{align*}
Together with \eqref{A-lower-bound}-\eqref{defect-bounds}
and $\partial_tZ=-f'(t)/d$ yields
\begin{align}\label{magnetic-correction-time-estimate}
 \left|\partial_t\left[
 \big(\rho_\e-H(z-f(t))\big)\partial_z\mathbf B^s
 +\frac{c^s}{A_\e}\boldsymbol{\mathcal D}_\e
 \right]\right|
 \leq Cd^{-2}L_\e^Ne^{-c|Z|}.
\end{align}
For $z\ne f(t)$, the Heaviside function $H(z-f(t))$ is constant.  Since the background magnetic shear satisfies
$(\partial_t-\partial_z^2)\partial_z\mathbf B^s=0$, \eqref{magnetic-correction-spatial-estimates}-\eqref{magnetic-correction-time-estimate} yield
\begin{align}\label{magnetic-parabolic-pointwise}
 \left| (\partial_t-\partial_z^2)
 \boldsymbol{\mathcal B}_\e\right|
 =\left| (\partial_t-\partial_z^2)\left[
 \big(\rho_\e-H(z-f(t))\big)\partial_z\mathbf B^s
 +\frac{c^s}{A_\e}\boldsymbol{\mathcal D}_\e
 \right]\right|
 \leq Cd^{-2}L_\e^Ne^{-c|Z|}.
\end{align}
In view of \eqref{bjump}, we get
%the piecewise regular function in \eqref{magnetic-parabolic-pointwise} is the full distributional parabolic derivative of $\boldsymbol{\mathcal B}_\e$; no singular measure remains.On the fixed cutoff support the weight $e^{\alpha z}$ is uniformly bounded, and $dz=d\,dZ$; therefore
\begin{align}
 \begin{aligned}
 \left\|(\partial_t-\partial_z^2)
 \boldsymbol{\mathcal B}_\e(t)\right\|_{L_\alpha^2}^2
 &\leq C d^{-4}L_\e^{2N}
 \int e^{-2c|Z|}\,dz\\
 &=C d^{-3}L_\e^{2N}
 \int e^{-2c|Z|}\,dZ.
 \end{aligned}
\end{align}
Taking the square root and then using \eqref{R34} and
\eqref{E-two-sided}, we finally obtain
\begin{align}
 \|(R_{3\e},R_{4\e})(t)\|_{L_\alpha^2}
 &=|E_\e(t)|\,
 \left\|(\partial_t-\partial_z^2)
 \boldsymbol{\mathcal B}_\e(t)\right\|_{L_\alpha^2}\notag\\
 &\leq Cd^{-3/2}L_\e^N
 e^{\frac{\sigma_0t}{d^2}}.\label{induction-residual-bound}
\end{align}

Finally, \eqref{momentum-residual-bound} and
\eqref{induction-residual-bound}, with $d^{-3/2}=\e^{-3/8}$, prove
\eqref{residual-improved-mode}.
\end{proof}
%Therefore,
%This means that, when considering the linear equation \eqref{mhd2}, the constructed approximate solutions $(u_\e,v_\e,w_\e,b_\e,g_\e,h_\e)(t,x,y,z)$ satisfy the linearized equation \eqref{jsmhd2}, with the error terms $R_{1\e},R_{2\e},R_{3\e},R_{4\e}$  bounded by $C\e^{-\frac{1}{4}}e^{\frac{\sigma_0t}{\sqrt{\e}}}$.

\subsection{Proof of Theorem ~\ref{TH1}~}
%We will use the contradiction method from \cite{JAMS} to prove Theorem \ref{TH1}. We treat $\partial_{\mathfrak h}=\partial_y$; the $x$ case is identical after replacing the modal frequency $\e^{-1}$ by $|a|\e^{-1}$.
We first prove the assertion for \(\partial_h=\partial_y\). When \(a\neq0\), the proof for \(\partial_h=\partial_x\) is identical, with the modal frequency \(\varepsilon^{-1}\) replaced by \(|a|\varepsilon^{-1}\).
Let
\[
 \gamma_{\mathfrak h}:=
 \begin{cases}1,&\partial_{\mathfrak h}=\partial_y,\\ |a|,&\partial_{\mathfrak h}=\partial_x.
 \end{cases}
\]
Fix $0<\sigma<\sigma_0/\sqrt{\gamma_{\mathfrak h}}$ and suppose, contrary to
\eqref{bsdjl}, that for some $m\geq0$, $\mu<\frac18$ and $\delta>0$,
\begin{align}\label{bounded-semigroup-improved}
 \sup_{0\leq s<t\leq\delta}
 \big\|e^{-\sigma(t-s)\sqrt{|\partial_{\mathfrak h}|}}\mathcal T(t,s)\big\|_{
 \mathcal L(H^m,H^{m-\mu})}<\infty.
\end{align}
%Conjugating the full four-component linearized operator by the periodic mode $e^{i\e^{-1}(y-ax)}$, with $\e$ restricted by \eqref{epsilon-sequence}, defines a modal evolution operator $\mathcal T_\e(t,s)$ .  Estimate \eqref{bounded-semigroup-improved} implies
From \eqref{bounded-semigroup-improved}, it is easy to see that
\begin{align}\label{modal-semigroup-improved}
 \|\mathcal T_\e(t,s)\|_{\mathcal L(L_\alpha^2)}
 \leq C\e^{-\mu}
 \exp\!\left(\frac{\sigma\sqrt{\gamma_{\mathfrak h}}(t-s)}{\sqrt\e}\right).
\end{align}

Let $\mathbf Z_\e=(U_\e,V_\e,B_\e,G_\e)$ be the approximate tangential profile of
Proposition~\ref{prop-improved-residual}, and let $\mathbf Z$ be the exact modal
solution with initial value $\mathbf Z_\e(0)$.  Duhamel's formula, with the harmless
sign of the residual absorbed into $\mathbf R_\e$, gives
\[
 \mathbf Z(t)-\mathbf Z_\e(t)
 =-\int_0^t\mathcal T_\e(t,s)\mathbf R_\e(s)\,ds.
\]
Combining \eqref{modal-semigroup-improved} with
\eqref{residual-improved-mode} and using
$\sigma\sqrt{\gamma_{\mathfrak h}}<\sigma_0$, we obtain
\begin{align}\label{duhamel-improved}
 \|\mathbf Z(t)-\mathbf Z_\e(t)\|_{L_\alpha^2}
 &\leq C\e^{-\mu-3/8}L_\e^N
 \int_0^t e^{\sigma\sqrt{\gamma_{\mathfrak h}}(t-s)/\sqrt\e}
 e^{\sigma_0s/\sqrt\e}\,ds\\
 &\leq C\e^{1/8-\mu}L_\e^N e^{\sigma_0t/\sqrt\e}.
\end{align}
Since $\mu<\frac18$, the prefactor in the last line tends to zero.

Choose
\begin{align}\label{logarithmic-time}
 t_\e=K\sqrt\e\,|\log\e|,
 \qquad
 K\big(\sigma_0-\sigma\sqrt{\gamma_{\mathfrak h}}\big)>\mu+1.
\end{align}
Apply Proposition~\ref{prop-improved-residual} with a fixed $K_0>K$.  For small
$\e$, $t_\e<\min\{\delta,K_0\sqrt\e L_\e\}$.
Equations \eqref{growth-improved-mode} and \eqref{duhamel-improved} then yield
\begin{align}\label{exact-lower-improved}
 \|\mathbf Z(t_\e)\|_{L_\alpha^2}
 \geq \frac{c_0}{2}e^{\sigma_0t_\e/\sqrt\e}.
\end{align}
On the other hand, \eqref{modal-semigroup-improved} and the initial bound in
Proposition~\ref{prop-improved-residual} give
\begin{align}\label{exact-upper-improved}
 \|\mathbf Z(t_\e)\|_{L_\alpha^2}
 \leq C\e^{-\mu}L_\e^N
 e^{\sigma\sqrt{\gamma_{\mathfrak h}}t_\e/\sqrt\e}.
\end{align}
The choice \eqref{logarithmic-time} makes \eqref{exact-upper-improved} incompatible
with \eqref{exact-lower-improved} as $\e\to0$.  This proves \eqref{bsdjl}.

\subsection*{Acknowledgements}
%Zhonger Wu is supported by the China Postdoctoral Science Foundation under grant 2024M753434.
Zhonger Wu is supported by National Natural Science Foundation of China Grant 12601431 and STU Scientific Research Initiation Grant NTF25028T.

\end{document}